\documentclass[12pt, reqno]{amsart}

\usepackage[margin=1in]{geometry}
\usepackage{color}
\usepackage{amsfonts}
\usepackage{amssymb}
\usepackage{amsmath}
\usepackage{amsthm}
\usepackage{latexsym}
\usepackage{graphicx}
\usepackage{soul}
\usepackage{hyperref}
\usepackage{tikz}
\hypersetup{colorlinks,citecolor=red,pdfstartview=FitH, pdfpagemode=None}

\numberwithin{equation}{section}

\newtheorem{theorem}{Theorem}[section]
\newtheorem{lemma}[theorem]{Lemma}

\theoremstyle{definition}

\newtheorem{example}[theorem]{Example}

\newtheorem{remark}[theorem]{Remark}
\newtheorem{question}[theorem]{Question}

\def\R{\mathbb R}
\def\C{\mathbb C}
\def\D{\mathbb D}
\def\H{\mathbb H}

\def\P{\mathbb P}

\def\Cnn{\C_{n\times n}}

\def\diag{{\mbox{diag}\,}}
\def\Diag{{\mbox{Diag}\,}}

\begin{document}

\title[Determinantal inequality]{Weak log majorization and determinantal inequalities}

\author[Tin-Yau Tam]{Tin-Yau Tam$^1$}
\address{$^1$ Department of Mathematics and Statistics,  Auburn University, Auburn, AL 36849,  USA}
\email{tamtiny@auburn.edu}

\author[Pingping Zhang]{Pingping Zhang$^2$$^*$}
\address{$^2$ School of Science\\
Chongqing University of Posts and Telecommunications\\
Chongqing, 400065\\
China
}
\email{zhpp04010248@126.com}

\thanks{$^*$Corresponding author.}




\begin{abstract}Denote by $\P_n$ the set of $n\times n$ positive definite matrices.
Let $D = D_1\oplus \dots \oplus D_k$, where $D_1\in \P_{n_1}, \dots, D_k \in \P_{n_k}$ with $n_1+\cdots + n_k=n$. Partition $C\in \P_n$ according to $(n_1, \dots, n_k)$ so that $\Diag C = C_1\oplus \dots \oplus C_k$. We prove the following weak log majorization result:
\begin{equation*}
\lambda (C^{-1}_1D_1\oplus \cdots \oplus C^{-1}_kD_k)\prec_{w \,\log}
\lambda(C^{-1}D),
\end{equation*}
where $\lambda(A)$ denotes the vector of eigenvalues of $A\in \Cnn$.
The inequality does not hold if one replaces the vectors of   eigenvalues  by the vectors of singular values, i.e.,
\begin{equation*}
s(C^{-1}_1D_1\oplus \cdots \oplus C^{-1}_kD_k)\prec_{w \,\log}
s(C^{-1}D)
\end{equation*}
is not true. As an application,
we provide a generalization of  a determinantal inequality of Matic \cite[Theorem 1.1]{M}. In addition, we obtain a weak  majorization result which is complementary to a determinantal inequality of Choi \cite[Theorem 2]{C} and give a weak log majorization open question.

\end{abstract}
\maketitle
\section{Introduction}

Denote by $\Cnn$ the set of $n\times n$ complex  matrices and $\P_n\subset \Cnn$ the set of $n\times n$ positive definite matrices.
For $A\in \Cnn$, we denote by $A^{*}$ and $|A|=(A^{*}A)^{\frac{1}{2}}$ the conjugate transpose and the positive semidefinite part of $A$,  respectively.  Given $n\times n$  Hermitian matrices $A$ and $B$,  $A\leq B$ means that   $B-A$  is positive semidefinite.

For $x = (x_1, x_2, \dots, x_n)$, $y = (y_1, y_2, \dots, y_n)\in\R^n$, let $x^{\downarrow} = (x_{[1]}, x_{[2]}, \dots, x_{[n]})$ denote a rearrangement of the components of $x$ such that $x_{[1]} \ge  x_{[2]} \ge \cdots \ge x_{[n]}$. The notation $x\leq y$ means that $x_{[i]}\leq y_{[i]},\,\,i=1,\ldots, n.$  We say that $x$ is weakly majorized by $y$, denoted by $x\prec_{w}y$, if $\sum_{j=1}^{k}x_{[j]}\leq \sum_{j=1}^{k}y_{[j]}$ for all $1\leq k\leq n$. We say that $x$ is majorized by $y$, denoted by $x\prec y$, if  $x\prec_w y$ and $\sum_{j=1}^{n}x_{j}=\sum_{j=1}^{n}y_{j}$.

Let $\R_+$ denote the set of all positive real numbers and $\R^n_+ = (\R_+)^n$. Given $x, y\in (\R_+)^n$, we say that $x$ is weakly log-majorized by $y$, written as $x\prec_{w\,\log}y$, if $\prod_{i=1}^{k}x_{[i]}\leq \prod_{i=1}^{k}y_{[i]}$, for $k=1,\ldots, n$;  $x$ is log-majorized by $y$, denoted by $x\prec_{\log}y$, if $x\prec_{w\,\log}y$ and  $\prod_{j=1}^{n}x_{j}=\prod_{j=1}^{n}y_{j}$.

Let $A\in \P_n$. Denote by $\lambda(A) = (\lambda_1(A), \dots, \lambda_n(A))\in \R_+^n$  the vector of eigenvalues of $A$ and we may arrange the eigenvalues in non-increasing order $\lambda_1(A) \ge \cdots \ge \lambda_n(A)$.

Matic \cite[Theorem 1.1]{M} proved the following determinantal inequality. Zhang \cite{Z} and Choi \cite {C} gave two different proofs, respectively. We state the theorem using Choi's version.

\begin{theorem}\rm (Matic \cite{M}) \label{Matic}
Let $C\in \P_n$ and $D = D_1\oplus \dots \oplus D_k$, where $D_1\in \P_{n_1}, \dots, D_k \in \P_{n_k}$ with $n_1+\cdots + n_k=n$.  Partition $C$ according to $(n_1, \dots, n_k)$ so that $\Diag C = C_1\oplus \dots \oplus C_k$ in which $\Diag C$ is the main block diagonal of $C$. Then
\begin{equation}\label{MC}
 \det (I_{n_1}+ C_1^{-1}D_1) \cdots  \det (I_{n_k}+ C_k^{-1}D_k)\leq \det (I_n+ C^{-1}D).
\end{equation}
\end{theorem}
In this paper we obtain the following weak log majorization result where $C$ and $D$ are given as in Theorem \ref{Matic}:
\begin{equation}\label{TP}
\lambda (C^{-1}_1D_1\oplus \cdots \oplus C^{-1}_kD_k)\prec_{w \,\log}
\lambda(C^{-1}D).
\end{equation}
We note that \eqref{TP} does not hold if one replaces the vectors of eigenvalues  by the vectors of singular values.

We would like to  point out that \eqref{TP} is  more fundamental than \eqref{MC}. In order to see that we will first derive the following determinantal inequality as an application of \eqref{TP}:
\begin{equation}\label{P}
 \det (I_{n_1}+ (C_1^{-1}D_1)^{p}) \cdots  \det (I_{n_k}+ (C_k^{-1}D_k)^{p})\leq \det (I_n+ (C^{-1}D)^{p}), \quad p\ge 0,
\end{equation}
which is evidently a generalization of  \eqref{MC}.
We will show by example that \eqref{P} is not true when $p<0$.

By looking at \eqref{P} as a generalization of  \eqref{Matic}, one might ask whether the  following two possible generalizations of \eqref{Matic} are true or not:
\begin{equation}\label{false1}
\det (I_{n_1}+ |C_1^{-1}D_1|^{p}) \cdots  \det (I_{n_k}+ |C_k^{-1}D_k|^{p})\leq\det (I_n+ |C^{-1}D|^{p}), \quad   p\ge 0,
\end{equation}
and
\begin{equation}\label{false2}
 \det (I_{n_1}+ C_1^{-p}D_1^{p}) \cdots  \det (I_{n_k}+ C_k^{-p}D_k^{p})\leq\det (I_n+ C^{-p}D^{p}), \quad  p\ge 0.
\end{equation}
Both answers are negative and we will provide a counterexample for both inequalities.

Choi \cite[Theorem 2]{C} obtained the following determinantal inequality:
\begin{theorem}\rm (Choi \cite{C})\label{Cthm} Let $A_{i}\in \P_{n},\,i=1,\dots, m$, and $\Diag A_i = A_i^{(1)} \oplus \cdots \oplus A_i^{(k)}$,
where  $A_{i}^{(j)}\in \P_{n_j}$ for $i=1, \dots, m$, $j=1,\ldots, k$.
\begin{equation}\label{Choi}
\det\left(\sum_{i=1}^{m}(A_{i}^{(1)})^{-1}\right)\cdots\det\left(\sum_{i=1}^{m}(A_{i}^{(k)})^{-1}\right)\leq\det\left(\sum_{i=1}^{m}A_{i}^{-1}\right).
\end{equation}
\end{theorem}
We present a weak  majorization inequality which is complementary to \eqref{Choi} and pose a weak log majorization open problem.

\section{Weak log majorization and generalization of Matic's determinantal inequality}

Regarding \eqref{P}, one may ask whether $(C^{-1}D)^{p}$ is well defined or not when $p\in \R$. The question makes sense  as we know that  $X^p$ is not necessarily defined for a general matrix $X\in \Cnn$. However, we can define $X^p$
when $X$ is hyperbolic, i.e., $X$ is diagonalizable with positive eigenvalues.
Let $\H_n$ be the set of all hyperbolic matrices in $\Cnn$  and $\D_n$ be the set of all diagonalizable matrices in $\Cnn$ with real eigenvalues.
We claim that the exponential map $\exp: \D_n\to \H_n$ is bijective. It is surjective since each $X\in \H_n$ can be written as $X  = SDS^{-1}$ for some nonsingular $S$, where $D = \diag (e^{\lambda_1}, \dots, e^{\lambda_n})$ and $\lambda_1, \dots, \lambda_n\in \R$. Set $Y:=S(\log D)S^{-1}$,
where $\log D: =  \diag (\lambda_1, \dots, \lambda_n)$. Note that
$$e^Y= e^{S(\log D)S^{-1}} = S(e^{\log D})S^{-1}= SDS^{-1}= X.$$ We are going to show that $\exp$ is injective.
Let $Y_1, Y_2\in \D_n$  such that $e^{Y_1}=e^{Y_2}$. Let $P_1D_1P_1^{-1} = Y_1$ and $P_2D_2P_2^{-1} = Y_2$, where $D_1$ and $D_2$ are diagonal with diagonal entries arranged in non-increasing order. Taking exponentials of both sides yields $P_1e^{D_1}P_1^{-1} = P_2e^{D_2}P_2^{-1}$ and thus
$D_1=D_2=D$, say, by spectrum consideration. Write $D:=d_1I_{n_1}\oplus \cdots \oplus d_kI_{n_k}$, where $d_1 > \cdots >  d_k$.  Set $Q = P_1^{-1}P_2$. So
$e^DQ = Qe^D$ and hence $Q = Q_1\oplus \cdots \oplus Q_k$, where $Q_i\in \C_{n_i\times n_i}$. Thus $DQ=QD$, i.e.,
$DP_1^{-1}P_2 =  P_1^{-1}P_2D$, which implies that $Y_1=Y_2$. So we conclude that the map $\exp: \D_n\to \H_n$ is bijective. Now given $X\in \H_n$,
define $X^p:=e^{pY}$, $p\in \R$, where $Y$ is the unique matrix in $\D_n$ such that $X=e^Y$. Explicitly,
if we write $X=S\diag(\xi_1, \dots, \xi_n)S^{-1}$ for some nonsingular $S\in \Cnn$, where $\xi_1, \dots, \xi_n\in \R_+$ are the eigenvalues of $X$ in any order, then
$X^p := S\diag (\xi_1^p, \dots, \xi_n^p)S^{-1}$ since $Y =S\diag (\log \xi_1, \dots, \log \xi_n)S^{-1}$.

Let us get back to the well-definedness of $ (C^{-1}D)^{p}$, $p\in \R$, regarding \eqref{P}. Note that $C^{-1}, D\in \P_n$.
Given $A, B\in \P_n$, the product $AB$ is hyperbolic though it may not be in $\P_n$. It is because that $AB$ is similar to $B^{1/2}ABB^{-1/2} = B^{1/2}AB^{1/2}\in \P_n$, which is unitarily similar to a positive diagonal matrix. Hence  $AB = P\diag(\gamma_1, \dots, \gamma_n)P^{-1}$, for some nonsingular $P\in \Cnn$, where $\gamma_1, \dots, \gamma_n\in \R_+$ so that $(AB)^p = P\diag(\gamma_1^p, \dots, \gamma_n^p)P^{-1}$, $p\in \R$.

Now we give our main result as follows.
\begin{theorem}\label{main}\rm
Under the conditions as in Theorem \ref{Matic}, we have
\begin{equation} \label{weak log}
\lambda (C^{-1}_1D_1\oplus \cdots \oplus C^{-1}_kD_k)\prec_{w \,\log}
\lambda(C^{-1}D).
\end{equation}
\end{theorem}
\begin{proof} 
It suffices to prove the case when $D=I_n$:
 \begin{eqnarray}\label{simple1}
 \lambda(C_{1}^{-1}\oplus\cdots\oplus C_{k}^{-1})\prec_{w\, \log}\lambda(C^{-1}).
 \end{eqnarray}
The reason is that
$$
\lambda(C^{-1}D) = \lambda( (D^{-1/2}CD^{-1/2})^{-1})
$$
 and
 $$\Diag(D^{-1/2}CD^{-1/2})=(D_{1}^{-1/2}C_{1}D_{1}^{-1/2})\oplus\cdots\oplus(D_{k}^{-1/2}C_{k}D_{k}^{-1/2}).$$
 Thus
 \eqref{weak log} is equivalent to
 $$
\lambda ((\Diag (D^{-1/2}CD^{-1/2}))^{-1})\prec_{w \,\log} \lambda ((D^{-1/2}CD^{-1/2})^{-1})
 $$
so it is sufficient to show \eqref{simple1}.
By a result of Ky Fan  \cite [p.308]{MOA}  we have
\begin{eqnarray}\label{majorization}
\lambda (C_1\oplus \dots \oplus C_k)\prec \lambda (C).
\end{eqnarray}
It is known that \cite[p.165]{MOA} for any convex function  $f:\R\to \R$, if $x\prec y$, where $x=(x_{1}, x_{2},\ldots,x_{n}), y=(y_{1}, y_{2},\ldots, y_{n})\in \mathbb{R}^{n}$, then
\begin{equation*}
(f(x_{1}), f(x_{2}), \ldots, f(x_{n}))\prec_{w} (f(y_{1}), f(y_{2}), \ldots, f(y_{n})).
\end{equation*}
Since $f(x)=x^{-p}$, where $p>0$, is convex on $(0, +\infty)$,  by \eqref{majorization}, we have
 \begin{eqnarray}\label{ineq30} \lambda(C_{1}^{-p}\oplus\cdots\oplus C_{k}^{-p})\prec_{w}\lambda(C^{-p}), \quad p>0.
 \end{eqnarray}
For $m$ positive numbers $a_1, \dots, a_m$,  their power mean  of order $r$ is defined by
$$M_{r}(a):= \left(\frac{1}{m}\sum_{i=1}^{m}a_i^{r}\right)^{\frac{1}{r}}$$
and it is known that $\lim_{r\to 0^+}M_{r}(a)= (\prod_{i=1}^{m}a_{i})^{\frac{1}{m}}$ (see \cite[p.15]{HLP}). Applying this to \eqref{ineq30} leads to
$$\lambda (C^{-1}_{1}\oplus\cdots\oplus C^{-1}_{k})\prec_{w \,\log}\lambda(C^{-1}),$$
i.e., \eqref{simple1}.
Thus, we complete the proof.
\end{proof}

 We now give an example to show that  \eqref{weak log} may not be true if $D\in \P_n$ is  not in diagonal block form, where
  $\Diag D = D_1\oplus \dots \oplus D_k$, and  $D_1\in \P_{n_1}, \dots, D_k \in \P_{n_k}$ with $n_1+\cdots + n_k=n$.
\begin{example}\label{WXZ4} Let $$C =
    \begin{pmatrix}
    14& 8& 9& 8\\
     8& 12& 7& 7\\
      9& 7& 10& 8\\
       8& 7& 8& 8
  \end{pmatrix},\quad
C_{1}= \begin{pmatrix}
14& 8\\
     8& 12
  \end{pmatrix},\quad
C_{2} =
    \begin{pmatrix}
       10& 8\\
      8& 8
  \end{pmatrix},$$
  $$
D=   \begin{pmatrix}
11& 12& 6& 11\\
 12& 16& 7& 12\\
  6& 7& 5& 6\\
   11& 12& 6& 14
  \end{pmatrix},
\quad D_{1} =
    \begin{pmatrix}
       11& 12\\
       12& 16
  \end{pmatrix},\quad
D_{2}=   \begin{pmatrix}
 5& 6\\
 6& 14
  \end{pmatrix}.
  $$
  By Matlab,
  \begin{eqnarray*}
  \lambda(C^{-1}D) &=&(4.8921,
    1.0664,
 0.3433,
    0.1772),\\
  \lambda(C_{1}^{-1}D_{1}) &=&(1.3488, 0.2281),\\
  \lambda(C_{2}^{-1}D_{2}) &= &(4.8080, 0.4420).
  \end{eqnarray*}
  So $\lambda(C_{1}^{-1}D_{1}\oplus C_{2}^{-1}D_{2})\not \prec_{w\,\log} \lambda(C^{-1}D)$ by considering the products of the first two largest entries of both sides.
\end{example}

\begin{remark}
One cannot replace the weak log majorization by the log majorization in \eqref{weak log}. For example let $C, D_{1}, D_{2}$ be given in  Example \ref{WXZ4}, and $D:=D_{1}\oplus D_{2}$.
Note that $$\lambda(C_{1}^{-1}D_{1}\oplus C_{2}^{-1}D_{2}) \not \prec_{\log} \lambda(C^{-1}D)$$ as $\det (C_1^{-1}D_1)\det (C_2^{-1} D_2)= 0.6538  \not= 2.1717=\det (C^{-1}D)$ by Matlab.
\end{remark}
\begin{remark}  Taking $D=I_{n}$ in Theorem \ref{main}, we easily obtain
\begin{equation}\label{rem}\prod_{i=m}^{n}\lambda_{i}(C)\leq\prod_{i=m}^{n}\lambda_{i}(C_{1}\oplus\cdots\oplus C_{k}),\,\,\,m=1,\ldots,n.
\end{equation}
\eqref{rem} implies Fischer inequality.

\end{remark}

\begin{remark} Let $C\in \P_n$ be given as in Theorem \ref{Matic} and $D\in \Cnn$ be positive semidefinite with $\Diag D =D_{1}\oplus\cdots\oplus D_{k}$. Wang, Xi, and Zhang \cite [Theorem 4] {WXZ} proved
$$\lambda(C_{1}^{-1}D_{1}\oplus 0)\le \lambda(C^{-1}D).$$
In Example \ref{WXZ4},  one can see that $(\lambda(C_{i}^{-1}D_{i}), 0, 0)\le  \lambda(C^{-1}D)$, $i=1, 2$.   Theorem \ref{main} complements this result when $D\in \P_n$ is in diagonal block form.
\end{remark}

Next we give an application of Theorem \ref{main} as follows.
\begin{theorem}\label{Zhang}\rm
Under the conditions as in Theorem \ref{Matic}, we have
\begin{equation} \label{det1}
\det (I_{n_1}+ (C_1^{-1}D_1)^p) \cdots  \det (I_{n_k}+ (C_k^{-1}D_k)^p)\leq \det (I_n+ (C^{-1}D)^p),\quad p\geq0.
\end{equation}
\end{theorem}
\begin{proof}When $p=0$, it is trivial so we may assume $p>0$.
By Theorem \ref{main}, we have
\begin{equation}\label{2.7}
\log\lambda (C^{-1}_{1}D_{1}\oplus\cdots\oplus C^{-1}_{k}D_{k})\prec_{w}\log\lambda(C^{-1}D).
\end{equation}
Let us recall a known result  \cite[p.167]{MOA} for any increasing convex function  $f:\R\to \R$. If $x\prec_{w} y$, where $x=(x_{1}, x_{2},\ldots,x_{n}), y=(y_{1}, y_{2},\ldots, y_{n})\in \mathbb{R}^{n}$, then
\begin{equation*}
(f(x_{1}), f(x_{2}), \ldots, f(x_{n}))\prec_{w} (f(y_{1}), f(y_{2}), \ldots, f(y_{n})).
\end{equation*}
Since the function $f(x)=\log(1+e^{px})$, where $p>0$, is an increasing convex function,  applying the function to \eqref{2.7} gives
\begin{equation*}
\sum_{i=1}^{n}\log(1+(\lambda_{i}(C^{-1}_{1}D_{1}\oplus\cdots\oplus C^{-1}_{k}D_{k}))^{p})\leq \sum_{i=1}^{n}\log(1+(\lambda_{i}(C^{-1}D))^{p}),
\end{equation*}
i.e.,
\begin{equation*}
\sum_{i=1}^{n}\log(1+\lambda_{i}((C^{-1}_{1}D_{1})^{p}\oplus\cdots\oplus (C^{-1}_{k}D_{k})^{p}))\leq \sum_{i=1}^{n}\log(1+\lambda_{i}((C^{-1}D)^{p})).
\end{equation*}
The desired result follows by taking exponential on both sides.
\end{proof}
\begin{remark} When $p=1$, Theorem \ref{Zhang} reduces to Theorem \ref{Matic}.
\end{remark}

In the next example we show that \eqref{det1} is not true when $p<0$.
\begin{example}\label{counter1}
Let $n=2$, $n_1=n_2=1$, $D=I_2$,
    $$C =
    \begin{pmatrix}
    3   &2 \\2&3
  \end{pmatrix}\in \P_{2\times 2},\quad
C_1=   C_2= 3.
  $$
  Direct computation gives $\lambda(C) = \{5,1\}$.
  Let $p<0$ and set $q:=-p$ so $q>0$. Then
  $$
 \lambda (C^q) = \{5^q, 1\}, \quad  C_1^q=C_2^q=3^q.
  $$
Let
 $$ f(q):=\det (I_2+C^q)=(1+5^q)(1+1)  = 2+2\cdot 5^q$$ and $$g(q):=\det (1+C_1^q) \det (1+C_2^q) = (1+3^q)^2 = 1+2\cdot3^q+3^{2q}.$$
 We are going to show that $g(q)>f(q)$ for all $q>0$. Let $f(x)=2\cdot 3^{x}+3^{2x}-2\cdot 5^{x}-1$. Since
 $$
 f'(x)=(3^{x}+9^{x})\ln9-2\cdot5^{x}\ln5\geq 2 \sqrt{3^{x}\cdot 9^{x}}\ln9-2\cdot 5^{x}\ln5>0,\quad \mbox{for } x>0,
 $$
 we have $f(x)>f(0)=0$ when $x>0$.
 Thus \eqref{det1} is not true when $p<0$.
\end{example}

We would like to point out that \eqref{det1} is no longer true if $D\in \P_n$ is  not in diagonal block form. We give an example to show this as follows.
  \begin{example}
Let $p=1$, $n=2$, $n_1=n_2=1$,
 $$C =
    \begin{pmatrix}
    12& 7\\
      7& 10
  \end{pmatrix},\quad
C_{1}=12,\quad
C_{2} =10,$$
  $$
D=   \begin{pmatrix}
 16& 7\\
7& 5
  \end{pmatrix},
\quad
D_{1} = 16,\quad
D_{2}=  5.
  $$
Direct computation gives
$$\det(I_2+C^{-1}D)\approx3.1549< 3.5=\det(1+C_{1}^{-1}D_{1})\det(1+C_{2}^{-1}D_{2}).$$
\end{example}

Next we will show that  inequalities  \eqref{false1} and \eqref{false2} are not true and
we first give a  counterexample to the following inequality.
\begin{eqnarray}\label{counter}  \det (D_1^{-2}+C_1^{-2}) \cdots  \det (D_k^{-2}+C_k^{-2})\leq \det (D^{-2} +C^{-2}),
\end{eqnarray}
where $C, D, C_{i}, D_{i}, i=1,\ldots, k$, are given as in Theorem \ref{Matic}.
\begin{example}\label{example}
Let  $$C =
    \begin{pmatrix}
    16.25& 21& 10& 12.5\\
     21& 39.75& 20.75& 28.5\\
      10& 20.75& 22.5& 27.75\\
       12.5& 28.5& 27.75& 39.25
  \end{pmatrix},\quad
C_{1}= \begin{pmatrix}
16.25& 21\\
     21& 39.75
  \end{pmatrix},
  $$
 $$C_{2} =
    \begin{pmatrix}
       22.5& 27.75\\
      27.75& 39.25
  \end{pmatrix},\quad
D=   \begin{pmatrix}
14.7& 15& 0& 0\\
 15& 15.8& 0& 0\\
  0& 0& 0.25& 0.4\\
   0& 0& 0.4& 0.8
  \end{pmatrix},
  $$
and
$$D_{1} =
    \begin{pmatrix}
       14.7& 15\\
       15& 15.8
  \end{pmatrix},\quad
D_{2}=   \begin{pmatrix}
 0.25& 0.4\\
 0.4& 0.8
  \end{pmatrix}.
  $$
By Matlab
$$\det (D^{-2} +C^{-2})=51.0669<54.6523 =\det (D_1^{-2}+C_1^{-2})\det (D_2^{-2}+C_2^{-2}).$$
Therefore, \eqref{counter} is false.
\end{example}

Note that
\begin{eqnarray*}\det (I_n+ |C^{-1}D|^2)&= &\det (I_n+ DC^{-2}D)=\det(D(D^{-2}+C^{-2})D)\\
&=&\det(D^{-2}+C^{-2})\cdot (\det D)^2,
\end{eqnarray*}
and

$$\det (I_{n_1}+ |C_1^{-1}D_1|^2) \cdots  \det (I_{n_k}+ |C_k^{-1}D_k|^2)$$
\begin{eqnarray*}
&=&\det(D_{1}(D_{1}^{-2}+C^{-2})D_{1})\cdots \det(D_{k}(D_{k}^{-2}+C_{k}^{-2})D_{k})\\
&=&\det(D_{1}^{-2}+C_{1}^{-2})\cdot\det(D_{1})^{2}\cdots\det(D_{k}^{-2}+C_{k}^{-2})\cdot (\det D_{k})^2\\
&=&\det(D_{1}^{-2}+C_{1}^{-2})\cdots\det(D_{k}^{-2}+C_{k}^{-2})\cdot  (\det D)^2.
\end{eqnarray*}
Since  \eqref{counter} is false,  \eqref{false1} is invalid.

Note that  $$\det (I_n+ C^{-2}D^2)=\det((D^{-2}+C^{-2})D^{2})=\det(D^{-2}+C^{-2})\cdot (\det D)^2$$
and
$$\det(I_{n_1}+ C_1^{-2}D_1^2) \cdots  \det (I_{n_k}+ C_k^{-2}D_k^2)$$
\begin{eqnarray*}&=&\det(D_{1}^{-2}+C_{1}^{-2})\cdot\det(D_{1})^{2} \cdots  \det(D_{k}^{-2}+C_{k}^{-2})\cdot (\det D_{k})^2\\
&=&\det(D_{1}^{-2}+C_{1}^{-2}) \cdots  \det(D_{k}^{-2}+C_{k}^{-2})\cdot (\det D)^2.
\end{eqnarray*}
Since  \eqref{counter} is false,  \eqref{false2} is also invalid.

\begin{remark}
Since \eqref{false1} is not true, the vectors of eigenvalues cannot be replaced by the vectors of singular values in \eqref{weak log} from the proof in Theorem \ref{Zhang}. In other words,
\begin{equation*}
s(C^{-1}_1D_1\oplus \cdots \oplus C^{-1}_kD_k)\not\prec_{w \,\log}
s(C^{-1}D),
\end{equation*}
where $C, D, C_{i}, D_{i}, i=1,\ldots, k$, are given as in Theorem \ref{Matic}.
\end{remark}

\section{Weak  majorization complementary to Choi's determinantal inequality }
Let us first recall a result from \cite[Theorem 7.13]{Zhang}.
\begin{lemma}\label{lem}\rm Let $A\in \C_{n\times n}$ be positive semidefinite and let $[A]$ be the principal submatrix of $A$ corresponding to some fixed rows and columns. Assuming that the inverses involved exist, we have
  $$[A]^{-1}\leq[A^{-1}].$$
\end{lemma}

\begin{theorem}\rm\label{thm3.2}  Under the conditions as in Theorem \ref{Cthm}, we have
\begin{eqnarray}\label{3.0}\lambda\left(\sum_{i=1}^{m}(A_{i}^{(1)})^{-1}\oplus\cdots\oplus\sum_{i=1}^{m}(A_{i}^{(k)})^{-1}\right)^{p}&\prec_{w}&\lambda\left(\sum_{i=1}^{m}A_{i}^{-1}\right)^{p},\quad p\geq1.
\end{eqnarray}
\end{theorem}
\begin{proof} By \eqref{majorization}, we have
\begin{eqnarray} \label{3.1}
\lambda\left(\sum_{i=1}^{m}(A_{i}^{-1})^{(1)}\oplus\cdots\oplus\sum_{i=1}^{m}(A_{i}^{-1})^{(k)}\right)\prec\lambda\left(\sum_{i=1}^{m}A_{i}^{-1}\right).
\end{eqnarray}
By Lemma \ref{lem} and Weyl's monotonicity theorem \cite[p.63]{B}, we have
\begin{eqnarray} \label{3.2}
\lambda\left(\sum_{i=1}^{m}(A_{i}^{(1)})^{-1}\oplus\cdots\oplus\sum_{i=1}^{m}(A_{i}^{(k)})^{-1}\right)\leq\lambda\left(\sum_{i=1}^{m}(A_{i}^{-1})^{(1)}\oplus\cdots\oplus\sum_{i=1}^{m}(A_{i}^{-1})^{(k)}\right).
\end{eqnarray}
Now \eqref{3.1} and \eqref{3.2} lead to
\begin{eqnarray*}
\lambda\left(\sum_{i=1}^{m}(A_{i}^{(1)})^{-1}\oplus\cdots\oplus\sum_{i=1}^{m}(A_{i}^{(k)})^{-1}\right)\prec_{w}\lambda\left(\sum_{i=1}^{m}A_{i}^{-1}\right).
\end{eqnarray*}
The desired result follows by the fact that  $f(x)=x^{p},\,p>1,$ is increasing and convex on $(0, +\infty)$.
\end{proof}
\begin{remark} Choi's determinantal inequality \eqref{Choi} is equivalent to the following:
$$\det\left(\sum_{i=1}^{m}(A_{i}^{(1)})^{-1}\oplus\cdots\oplus\sum_{i=1}^{m}(A_{i}^{(k)})^{-1}\right)^{p}\leq\det\left(\sum_{i=1}^{m}A_{i}^{-1}\right)^{p},\quad p>0.$$
Thus, Theorem \ref{thm3.2} is a complement to Theorem \ref{Cthm}.
\end{remark}

It is natural to ask whether \eqref{3.0} holds for $0<p<1$? By Theorem \ref{thm3.2} and the proof of Theorem \ref{main}, this question is equivalent to the following weak log majorization question:
\begin{question}Under the conditions as in Theorem \ref{Cthm},
\begin{equation}\label{3.5}
\lambda\left(\sum_{i=1}^{m}(A_{i}^{(1)})^{-1}\oplus\cdots\oplus\sum_{i=1}^{m}(A_{i}^{(k)})^{-1}\right)\prec_{w \log}\lambda\left(\sum_{i=1}^{m}A_{i}^{-1}\right)?
\end{equation}
\end{question}
When $n=2, k=2, n_{1}=n_{2}=1$, \eqref{3.5} holds by \eqref{Choi} and \eqref{3.0}. The other cases are open.
We would like to point out that we performed computer experiments and the outcomes are consistent with the weak log majorization given in \eqref{3.5}.

\section*{Disclosure statement}
No potential conflict of interest was reported by the authors.

\section*{Acknowlegements}
The second author is supported by the National Natural Science Foundation of China (Grant numbers: 11601054, 11671060), the  Basic and Advanced Research Project of CQ CSTC (Grant number: cstc2016jcyjA0466), the Science and Technology Research Project of Chongqing Municipal Education Commission (Grant number: KJ201600644621) and the Youth Science Research Project of Chongqing University of Posts and Telecommunications (Grant number: A2014-104).

\end{document}